\documentclass[reqno,10pt]{amsart}
\usepackage{amscd,amssymb,amsmath,defs,color,url}
\usepackage{latexsym}
\usepackage[all]{xy}
\usepackage[colorlinks=true]{hyperref}

\def\tillam{{\wt\lam}}
\def\..{{,\dots,}}

\begin{document}

\author{Michael Temkin}
\title{Non-archimedean pinchings}

\address{Einstein Institute of Mathematics\\
               The Hebrew University of Jerusalem\\
                Edmond J. Safra Campus, Giv'at Ram, Jerusalem, 91904, Israel}
\email{michael.temkin@mail.huji.ac.il}
\date{\today}
\subjclass{Primary 14G22; Secondary 14D15}

\keywords{Non-archimedean analytic spaces, pinchings.}

\thanks{This research is supported by ERC Consolidator Grant 770922 - BirNonArchGeom.}

\begin{abstract}
We develop the theory of pinchings for non-archimedean analytic spaces. In particular, we show that although pinchings of affinoid spaces do not have to be affinoid, pinchings of Hausdorff analytic spaces always exist in the category of analytic spaces.
\end{abstract}

\maketitle

\section{Introduction}

\subsection{Pinchings of schemes}
Pushouts in algebraic geometry are subtle constructions which rarely exist. A classical pushout construction in algebraic geometry called {\em pinching} corresponds to the case of a diagram $X\stackrel{i}\hookleftarrow Y\stackrel{h}\to Y'$, where $i$ is a closed immersion and $h$ is finite. In particular, it is often used to obtain a non-normal scheme by {\em pinching} its normalization $X$ along a finite morphism $h$ from one of its closed subschemes $Y$. Moreover, Ferrand studied such pushouts even when $h$ is an arbitrary affine morphism with the following simple result serving as a starting point, see \cite[Lemme~1.3]{Fer}:

\begin{lem}\label{ferrlem}
Let $\phi\:A\to B$ and $\psi\:B'\to B$ be two ring homomorphisms with $A'=A\times_BB'$, and assume that $\phi$ is surjective. Then $\phi'\:A'\to B'$ is surjective, $\Ker(\phi')=\Ker(\phi)$ and $A\otimes_{A'}B'=B$. In addition, if $\psi$ is finite, then $\psi'\:A'\to A$ is finite.
\end{lem}

Furthermore, this construction is compatible with flat base changes and it follows easily, that passing to the spectra $X=\Spec(A)$, $Y=\Spec(B)$, $X'=\Spec(A')$ and $Y'=\Spec(B')$ one has $X'=X\coprod_YY'$ in the category of all schemes (or even stacks). More generally, a Ferrand's pushout datum consists of a closed immersion $i\:Y\into X$ and an affine morphism $h\:Y\to Y'$ of schemes, and it is called a {\em pinching datum} if $h$ is finite. Existence of Ferrand's pushout is very subtle and requires some assumptions. Ferrand described in \cite[Th\'eor\`eme~7.1]{Fer} a necessary condition for existence of the pushout in the category of schemes, and gave examples, when a pinching of projective varieties is not projective or even does not exist. It was shown in \cite{push} that usage of \'etale topology improves the situation: Ferrand's pushouts of schemes exist as algebraic spaces in a wide range of cases (conjecturally always), in particular, pinchings of schemes always exist. If the pushout exists, one has that $Y'\into X'$ is a closed immersion, $Y=X\times_{X'}Y'$, and in the case of pinchings $X\to X'$ is finite.

In general, Ferrand's pushout of varieties can be non-noetherian, but pinchings behave much better: by Artin-Tate lemma if $X$ is of finite type over a noetherian base $S$ (for example, a $k$-variety), then the same is true for $X'$.

\subsection{Liu's example}\label{liusec}
Let now $k$ be a complete non-archimedean field. In general, the theory of $k$-affinoid algebras and $k$-analytic (or rigid) spaces is quite analogous to the theory of affine algebras and algebraic varieties over a field. However, Liu showed in \cite{Liu} that a whole cluster of classical results does not extend to the $k$-analytic setting. Originally, all this was based on an example of an affinoid pinching datum $\calM(\calA)\hookleftarrow\calM(\calB)\to\calM(\calB')$ whose pinching $X'$ exists as an analytic space but is not affinoid. In particular, $\calA'=\calA\times_\calB\calB'$ is a non-affinoid subalgebra of $\calA$ such that $\calA'\into\calA$ is finite, and the non-affinoid space $X'$ has an affinoid (partial) normalization $X=\calM(\calA)$. Note also that an even simpler example of a non-affinoid pinching was constructed in \cite[Example~5.4]{descent}.

Liu's example implies that naive affinoid/analytic analogues of the following results fail: Artin-Tate lemma, affinoidness of fiber products for pinching type of data, the Serre's criterion of affineness with vanishing of cohomology. In a subsequent work \cite{liu-tohoku} Liu constructed a non-affinoid compact Stein domain in a two-dimensional disc, showing that Serre's criterion may fail even for analytic subdomains in a polydisc.

\subsection{Overview of the paper}
Recently a particular type of pinchings of analytic spaces was used in \cite{portayu} by Porta and Yue Yu in their work on cotangent complex of derived analytic spaces. The result really used in their application is correct, but the argument about existence of pinchings was flawed because of a use of affinoid Artin-Tate lemma. The goal of this paper is to clarify the situation with pinchings in the $k$-affinoid and $k$-analytic categories. Section~\ref{affsec} is devoted to a necessary material on affinoid algebras. We start with an affinoid version of Artin-Tate lemma and deduce a criterion when a fiber product $\calA'=\calA\times_\calB\calB'$ of affinoid algebras with a surjective $\phi\:\calA\to\calB$ and a finite $\calB'\to\calB$ is affinoid, see Theorem~\ref{fiberth}. It was observed in \cite{descent} that the constructions of non-affinoid products by Liu and in loc.cit. essentially use that the homomorphism $\phicirc$ (or $\tilphi$) is not surjective. Our result shows that this is not an accident: if the algebras are strict, $\calB$ is reduced and $\phicirc$ is surjective, then $\calA'$ is necessarily affinoid. However, the theorem is formulated more generally to also cover the important pinching case, when $\calB$ is not reduced. Finally, using this result and a descent argument, we show in Theorem~\ref{localpinch} that localizing $X=\calM(\calA)$ enough around $Y=\calM(\calB)$ one can achieve that the pushout is affinoid.

\begin{rem}
Geometrically this can be interpreted as follows: although $i\:Y\into X$ is a closed immersion of affinoid spaces, its affine formal model $\gti\:\Spf(\calAcirc)\to\Spf(\calBcirc)$ does not have to be a closed immersion in general. However, shrinking $X$ we can achieve that $\gti$ becomes a closed immersion, and then an affinoid pinching exists.
\end{rem}

In Section~\ref{pinchsec} we deal with the geometric pinching data. In fact, using the results of Section~\ref{affsec} the arguments are very close to the usual theory of pinchings of schemes. Moreover, analytic topology is fine enough, so it works as well as the \'etale topology. Thus, although the local theory is more pathological than in the case of schemes, the global theory is even nicer: a Hausdorff pinching data always possesses an analytic pushout and the construction is flat-local, see Theorem~\ref{pinchth} for the main result about pinchings.

Studying the failure of Serre's criterion belongs to the theory of Stein spaces and is not a goal of this paper, so we only briefly discuss it in Section~\ref{conjsec}. In the first version of this paper we suggested a conjecture that conceptually explains the differences between the categories of affine varieties and affinoid spaces. Soon after it was indeed proved by M. Xia in a work \cite{Xia} on Liu's spaces. In brief, the main new feature of the analytic category is that there exist non-affinoid Banach $k$-algebras which are locally affinoid, let us call them Liu algebras. Their spectra are precisely the compact Stein spaces, also called Liu spaces, i.e. the compact separated spaces that satisfy Serre's criterion. In other words, there exist ''non-classical'' affinoid objects in the $k$-analytic category whose algebras of functions are not of topologically finite type over $k$. In particular, examples of such objects can be obtained by pinching classical affinoids and first time they were discovered in this way.

\subsection{Conventiones}
By an {\em analytic field} we mean a field $K$ provided with a non-archimedean real valuation $|\ |\:K\to\bfR_{\ge 0}$ and with respect to which $K$ is complete. Throughout this paper $k$ is an analytic ground field and we work with $k$-analytic spaces as defined in \cite[\S1]{berihes}. If $l/k$ is an extension of analytic fields, $\calA$ is a Banach $k$-algebra and $X$ is a $k$-analytic space we will use notation $\calA_l=\calA\wtimes_kl$ and $X_l=X\times_kl$.

\subsection*{Acknowledgments}
The author is grateful to J\'{e}r\^{o}me Poineau, Antoine Ducros and Tony Yu Yue for useful discussions. Also, he is very grateful for the referee for careful reading of the first version of the paper and pointing out various inaccuracies.

\section{Affinoid algebras}\label{affsec}

\subsection{An affinoid version of Artin-Tate lemma}
A simple proof of the classical Artin-Tate lemma can be found in \cite[Tag 00IS]{stacks}. A similar argument with homogeneous elements proves the following graded version, see \cite[Th\'eor\`eme~2.7]{angelique} for details.

\begin{lem}\label{gradedAT}
Assume that $k$ is a noetherian graded ring and $A$ is a finitely generated graded $k$-algebra, which is finite over a graded $k$-subalgebra $B\subseteq A$. Then $B$ is finitely generated over $k$.
\end{lem}

Now, we can deal with the affinoid case -- see the converse implications below. A naive version of such a result fails due to Liu's example, so one has to add an assumption on the reductions (or something analogous). We formulate the result for general affinoid algebras; in the non-strict case this requires to work with the graded reduction $$\oplus_{r>0}\{a\in\calA|\ \rho_\calA(a)\le r\}/\{a\in\calA|\ \rho_\calA(a)<r\}$$ introduced in \cite{Temkin-local-properties} and denoted $\tilcalA_\gr$ below to distinguish it from the usual reduction, which is the homogeneous component of weight 1.

\begin{theor}\label{AT}
Assume that $\calA$ is a $k$-affinoid algebra, which is finite over one of its $k$-subalgebras $\calB$, and provide $\calB$ with the norm induced from $\calA$. Then

(i) $\calB$ is $k$-affinoid if and only if $\tilcalA_\gr$ is finite over $\tilcalB_\gr$.

(ii) If $\calA$ is strictly $k$-affinoid, then $\calB$ is $k$-affinoid if and only if $\tilcalA$ is finite over $\tilcalB$.
\end{theor}
\begin{proof}
The direct implications in (i) and (ii) follow from \cite[Proposition~3.1(iii)]{Temkin-local-properties} and \cite[6.3.5/1]{bgr}, respectively.

Conversely, consider the embedding $\phi\:\calB\into\calA$ and assume that $\tilphi_\gr$ is finite. Since $\calB\into\calA$ is an isometry, it is also an isometry with respect to the spectral seminorms, and hence $\tilphi_\gr$ is injective. So, $\tilcalB_\gr$ is finitely generated over $\tilk_\gr$ by the graded Artin-Tate lemma \ref{gradedAT}. Choose generators $\tilb_1\..\tilb_n$ of $\tilcalB_\gr$ and lift them to $\calB$. Set $r_i=\rho_\calB(b_i)$ and consider the homomorphism $\psi\:\calC=k\{r_1^{-1}T_1\..r_n^{-1}T_n\}\to\calA$ taking $T_i$ to $b_i$. Clearly, $\psi$ factors through $\calB$ and the map $\tilcalC_\gr\to\tilcalB_\gr$ is onto. Thus, $\tilcalA_\gr$ is finite over $\tilcalC_\gr$, and by \cite[Proposition~3.1(iii)]{Temkin-local-properties} $\psi$ is finite. Since $\calC$ is noetherian, the $\calC$-submodule $\calB$ of $\calA$ is finite and hence $k$-affinoid by \cite[Proposition~2.1.12]{berbook}. This proves the inverse implication in (i).

The inverse implication in (ii) is proved similarly, but one works with the usual reduction $\tilphi$ instead, uses the usual Artin-Tate and refers to \cite[6.3.5/1]{bgr} instead of \cite[3.1(iii)]{Temkin-local-properties}.
\end{proof}

\subsection{Fiber product}
Now, we give a criterion when a fiber product of affinoid algebras is affinoid.

\begin{theor}\label{fiberth}
Assume that $k$ is non-trivially valued and let $\psi\:\calB'\to\calB$ and $\phi\:\calA\onto\calB$ be homomorphisms of strictly $k$-affinoid algebras such that $\psi$ is finite admissible, $\phi$ is surjective and $\psi^\circ(\calB'^\circ)\subseteq\phi^\circ(\calA^\circ)$. Then the algebra $\calA'=\calA\times_\calB\calB'$ with the norm induced from $\calA\times\calB'$ is $k$-affinoid.
\end{theor}

%Before proving the theorem let us explain why such a little clumsy formulation is necessary in order to deal with the non-reduced case. First, the assumption that $\phi^\circ_l(\calB'^\circ_l)\subseteq\psi^\circ_l(\calA_l^\circ)$ is only used when $\calB$ is non-reduced because otherwise $\psi^\circ$ is automatically surjective by \cite[Theorem~6.3.5/1]{bgrbgr}. Second, if $\calA$ is reduced, but $\calB'$ is non-reduced and its radical is not mapped to zero by $\phi$, then it can never happen that $\phi^\circ_l(\calB'^\circ_l)\subseteq\psi^\circ_l(\calA_l^\circ)$. However, in such a case it still might be the case that $\calB'=\calC/I$ and the image of $\calCcirc$ in $\calB$ is contained in $\psi^\circ_l(\calA_l^\circ)$.

\begin{proof}
By the usual theory of pinchings, $\psi'\:\calA'\to\calA$ is finite and $\phi'\:\calA'\to\calB'$ is surjective. Therefore the embedding $\lam\:\calA'\into\calA\times\calB'$ is finite and in view of Theorem~\ref{AT}, it suffices to prove that $\tillam$ is finite. Thus, we should prove that the homomorphisms $\tilpsi'$ and $\tilphi'$ are finite.

Note that $\calA'^\circ=\calAcirc\times_{\calB}\calB'^\circ=\calAcirc\times_C\calB'^\circ$, where $C=\phicirc(\calAcirc)\subseteq\calBcirc$. The homomorphism $\calA'^\circ\to\calB'^\circ$ is surjective by Lemma~\ref{ferrlem}, hence $\tilphi'$ is surjective. The homomorphism $\calB'^\circ\to C$ is integral by \cite[Theorem~6.3.5/1]{bgr}, hence $C$ is a filtered union of finite $\calB'^\circ$-subalgebras $C_i$. Let $A_i\subseteq\calAcirc$ be the preimage of $C_i$. Then $\calA'^\circ$ is the filtered union of subalgebras $A_i\times_{C_i}\calB'^\circ$ and the homomorphisms $A_i\times_{C_i}\calB'^\circ\to A_i$ is finite by the usual theory of pinchings. Therefore $\calA'^\circ$ is integral over $\calA^\circ$. Passing to the reductions we obtain that the homomorphism of $\tilk$-algebras $\wt{\calA'}\to\tilcalA$ is integral, and since $\tilcalA$ is finitely generated over $\tilk$, it is, in fact, finite.
\end{proof}

\subsection{Relative maximal modulus principle}
The usual maximum modulus principle states that for a $k$-affinoid space $X=\calM(\calA)$ and an element $f\in\calA$ one has that $\rho(f)=\max_{x\in X}|f(x)|$. We will need the following relative version:

\begin{lem}\label{rellem}
Assume that $X=\calM(\calA)$ is an affinoid space with a Zariski closed affinoid subspace $Y=\calM(\calB)$ and $f\in\calA$ is a function on $X$. Set $r=\rho_\calB(f|_Y)$, then

(i) If $r>0$, then there exists a neighborhood $W$ of $Y$ such that $\max_{x\in W}|f(x)|=r$.

(ii) If $r=0$, then for any $\veps>0$ there exists a neighborhood $W$ of $Y$ such that $\max_{x\in W}|f(x)|<\veps$.
\end{lem}
\begin{proof}
Claim (ii) follows from the compactness of $Y$. To prove (i) it suffices to show that the affinoid domain $X'=X\{r^{-1}f\}$ is a neighborhood of $Y$ in $X$. Since $Y\to X$ is finite, $Y=\Int(Y/X)$ and by \cite[Proposition~2.5.8(iii)]{berbook} applied to the composition $Y\into X'\into X$ we obtain that $Y\subset\Int(X'/X)$. By \cite[Corollary~2.5.13(ii)]{berbook}, $\Int(X'/X)$ is the topological interior of $X'$  inside $X$, hence $X'$ is a neighborhood of $Y$.
\end{proof}

\subsection{Descent}
For a tuple $r=(r_1\..r_n)$ of positive numbers let $K_r$ be the completed fraction field of $k(t_1\..t_n)$ with the generalized Gauss norm given by $|t_i|=r_i$. Also, we call an extension of analytic fields $l/k$ {\em topgebraic} if it can be embedded into $\whka/k$. The following result seems to be well-known to experts, but it is not easy to find a precise reference, so we provide a proof.

\begin{lem}\label{affdescent}
Let $l$ be a topgebraic extension of some $K_r$ with $r_1\..r_n$ linearly independent over $\sqrt{|k^\times|}$. Then a Banach $k$-algebra $\calA$ is $k$-affinoid if and only if $\calA_l=\calA\wtimes_kl$ is $l$-affinoid.
\end{lem}
\begin{proof}
Only the descent result needs a proof. We can assume that $l$ is not trivially valued, as otherwise we can replace it by any $l_r$ with $r\neq 1$. We will prove in two steps that $\calA_r=\calA\wtimes_k K_r$ and $\calA$ are affinoid. The latter follows from \cite[Corollary 2.1.8]{berbook} by induction on $n$, hence it suffices to consider the case when $l/k$ is topgebraic and $k$ is non-trivially valued. Set $K=\whka$, then $\calA_K=\calA_l\wtimes_lK$ is $K$-affinoid, and it suffices to establish the descent from $K$. Note, that $k^s$ is dense in $K$ because the valuation is non-trivial.

Choose affinoid generators $f_1\..f_n\in\calA_K$, that is, fix a surjective homomorphism $\psi\:K\{r_1^{-1}t_1\..r_n^{-1}t_n\}\to\calA_K$ sending $t_i$ to $f_i$. Recall that by \cite[Proposition~2.1.7]{berbook} a small perturbation of this system is still a family of generators, hence we can assume that $f\subset \calA_F$ for a finite Galois extension $F/k$. Then $\psi$ is the base change of the homomorphism $F\{r^{-1}t\}\to\calA_F$, which is surjective since $\wtimes_FK$ is exact. So, $\calA_F$ is $F$-affinoid and then $\calA=(\calA_F)^G$ for $G=\Gal(F/k)$ is $k$-affinoid by \cite[Proposition~6.3.3/3]{bgr}.
\end{proof}

The lemma covers our needs, so we do not pursue the generality and only discuss it without proof.

\begin{rem}
(i) One can remove any assumption on $r$ in the lemma. In particular, one does not even have to assume that the tuple is finite. However, it is important for the proof that $K_r/k$ has an orthogonal Schauder basis.

(ii) The lemma does not hold for an arbitrary extension $l/k$ of analytic fields. For example, if $x\in\bfA^1_k$ is of type 4 and radius $r$, then $l=\calH(x)$ is not $k$-affinoid, while $l\wtimes_kl\toisom k\{r^{-1}t\}$ is $l$-affinoid.
\end{rem}

\subsection{Local pinchings}
Now we can prove the main result of Section~\ref{affsec}: after shrinking an affinoid pinching data around the closed subspace, one can achieve that the fiber product is affinoid.

\begin{theor}\label{localpinch}
Assume that $\phi\:\calA\onto\calB$ and $\psi\:\calB'\onto\calB$ are finite admissible homomorphisms of $k$-affinoid algebras and $\psi$ is even surjective, and let $f_1\..f_n$ a family of generators $\Ker(\phi)$. For any $\veps>0$ set $\calA_\veps=\calA\{\veps^{-1}f_1\..\veps^{-1}f_n\}$ and consider the natural homomorphism $\phi_\veps\:\calA_\veps\to\calB$. Then there exists $\veps_0>0$ such that the fiber product $\calA'_\veps=\calA_\veps\times_\calB\calB'$ is $k$-affinoid for $\veps\le\veps_0$.
\end{theor}
\begin{proof}
Choose $r=(r_1\..r_m)$ such that $r_1\..r_m$ are linearly independent over $|k^\times|$, $K_r$ is non-trivially valued and the algebras $\calA\wtimes K_r$ and $\calB\wtimes K_r$ are strictly $K_r$-affinoid. Set $l=\wh{K_r^a}$. Since the functor $\wtimes_k l$ is exact, $\calA'_l=\calA_l\times_{\calB_l}\calC_l$. Thus, in view of Lemma~\ref{affdescent} it suffices to prove that $\calA'_l$ is $l$-affinoid. This reduces the claim to the following case: $k$ is non-trivially valued and algebraically closed and the algebras are strictly $k$-affinoid.

Next, choose a surjective homomorphism of strictly affinoid algebras $\calB''\onto\calB'$ with a reduced $\calB''$ and set $\calA''_\veps=\calA_\veps\times_\calB\calB''$. Then the squares in the following diagram are cartesian
$$\xymatrix{
\calA''_\veps \ar[r]\ar[d]^{\phi''}& \calA'_\veps \ar[r]\ar[d]^{\phi'}& \calA_\veps \ar[d]^\phi\\
\calB'' \ar[r]& \calB' \ar[r]& \calB
}
$$
and hence $\calA'_\veps$ is a quotient ring of $\calA''_\veps$. Therefore, it suffices to prove that $\calA''_\veps$ is affinoid for a small $\veps$, and replacing $\calB'$ by $\calB''$ we can assume that $\calB'$ is reduced. In particular, by \cite[Theorem~6.4.3/1]{bgr} $\calB'^\circ$ is generated by finitely many elements $b'_1\..b'_m$ as an adic $\kcirc$-algebra.

For each $i$ choose a lift $a_i\in\calA$ of $b_i=\psi(b'_i)$. Since $\rho(b_i)\le 1$ and the closed subset $\calM(B)$ of $\calM(\calA)$ is the intersection of the subdomains $\calM(\calA_\veps)$, we obtain from Lemma~\ref{rellem} that taking $\veps$ small enough one can achieve that $\rho_{\calA_\veps}(a_i)\le 1$. This implies that $\psi(\calB'^\circ)\subseteq\phi(\calA_\veps^\circ)$ and hence $\calA'_\veps$ is affinoid by Theorem~\ref{fiberth}.
\end{proof}

\section{Pinchings}\label{pinchsec}

\subsection{Pinching data}
By a {\em pinching datum} in the $k$-analytic category we mean a diagram consisting of a closed immersion $i\:Y\into X$ and a finite morphism $h\:Y\to Y'$. Usually we will denote pinching datum using notations like $\calP=(Y;X,Y')$, $\calP_i=(Y_i;X_i,Y'_i)$, etc. A {\em morphism} of pinching data $f\:\calP_1\to\calP$ consists of morphisms $f_X\:X_1\to X$, $f_Y\:Y_1\to Y$, $f_{Y'}\:Y'_1\to Y'$ such that $i_1$ and $h_1$ are pullbacks of $i$ and $h$, respectively. A morphism is flat, smooth, embedding of a subdomain, etc., if it is so componentwise. In particular, a {\em subdomain} $\calP_0\subseteq\calP$ is a compatible family of subdomains in $X,Y$ and $Y'$. A covering of $\calP$ by some of its subdomains is {\em admissible} if it is so componentwise.

\subsection{Affinoid pushouts}
A pinching datum is {\em affinoid} if all its components are affinoid, say $X=\calM(\calA)$, $Y=\calM(\calB)$ and $Y'=\calM(\calB')$, and it is {\em strongly affinoid} if, in addition, the algebra $\calA'=\calA\times_\calB\calB'$ is affinoid. Similarly, for an affinoid datum $\calP_i$ we will denote the algebras $\calB_i;\calA_i,\calB'_i$, etc. In the above situation $X'=\calM(\calA')$ is the pushout in the category of $k$-affinoid spaces, so we will use the notation $X'=\coprod^\aff\calP$.

The notion of flat morphisms we will use in the sequel was introduced by Ducros in \cite[\S4.1]{flat}. Recall, that $f\:\calM(\calA)\to\calM(\calB)$ is {\em naively flat} if $\calB\to\calA$ is flat, and by \cite[Theorem~8.3.6]{flat} $f$ is flat if and only if $f_K$ is naively flat for any analytic extension $K/k$.

\begin{lem}\label{affpinch}
Assume that $\calP=(Y;X,Y')$ is a strongly affinoid pinching datum and $X'=\coprod^\aff\calP$, then

(i) $X\to X'$ is a finite morphism, $Y'\into X'$ is a closed immersion, $Y=X\times_{X'}Y'$, and $|X'|=|X|\coprod_{|Y|}|Y'|$ as topological spaces.

(ii) Pulling back an affinoid space $X'_0$ over $X'$ to the pinching datum $$\calP\times_{X'}X'_0=(Y\times_{X'}X'_0;X\times_{X'}X'_0,Y'\times_{X'}X'_0)$$ yields an equivalence between the categories of affinoid $X$-flat spaces and strongly affinoid $\calP$-flat pinching data over $\calP$. An opposite equivalence is given by the affinoid pushout construction.
\end{lem}
\begin{proof}
The homomorphisms of rings $\calA\onto\calB$, $\calA'\onto\calB'$, $\calA'\to\calA$ and $\calB'\to\calB$ in (i) are finite, hence all completed tensor products in (i) and (ii) coincide with the usual tensor products, and everything in (i) and (ii) follows from the usual theory of pinchings of affine schemes, except the following two issues: (a) on the nose one obtains equivalence of naively flat data, (b) the classical theory deals with the topological spaces of affine spectra, so one has to show separately that $|X'|=|X|\coprod_{|Y|}|Y'|$.

(a) The claim about flatness easily reduces to the naive flatness once we prove that for any analytic extension $K/k$ one has $\calA'_K=\calA_K\times_{\calB_K}\calB'_K$. But this is so because the functor $\wtimes_kK$ preserves exactness of the sequence of finite $\calA'$-modules $$0\to\calA'\to\calA\times\calB'\to\calB\to 0.$$

(b) Applying (ii) to an affinoid domain $X'_0\subset X'\setminus Y'$ we obtain that the pushout of $\calP\times_{X'}X'_0=(\emptyset;X\times_{X'}X'_0,\emptyset)$ is isomorphic to $X'_0$, that is, $X\times_{X'}X'_0=X'_0$. This implies that $U=X\setminus Y=X'\setminus Y'$ and we obtain that set-theoretically
$$|X|\coprod_{|Y|}|Y'|=\left(|Y|\coprod|U|\right)\coprod_{|Y|}|Y'|=|Y'|\coprod|U|=|X'|.$$ It remains to show that a subset $V\subseteq X'$ is open whenever its preimages in $X$ and $Y'$ are open, and this follows from the fact that $X\coprod Y'\to X'$ is a surjective continuous map of compact spaces.
\end{proof}

\subsection{Strongly affinoid covers}
The main technical result about pinchings is existence of strongly affinoid covers.

\begin{lem}\label{pinchcover}
Any Hausdorff pinching datum $\calP=(Y;X,Y')$ possesses an admissible strongly affinoid covering $\calP=\cup_i\calP_i$.
\end{lem}
\begin{proof}
Step 1. {\it It suffices to prove the result for a subdomain $\calP'=(Y;W,Y')$, where $W$ is a subdomain of $X$, which is a neighborhood of $Y$.} Indeed, assume that $\calP'$ possesses a strongly affinoid covering $\calP'=\cup_i\calP_i$, find an admissible affinoid covering $X\setminus Y=\cup_{j\in J}X_j$ and set $\calP_j=(\emptyset;X_j,\emptyset)$. Then it is easy to see that $\cup_{l\in I\cup J}X_l$ is an admissible covering of $X$ and hence $\calP=\cup_{l\in I\cup J}\calP_l$ is an admissible strongly affinoid covering.

Step 2. {\it The case of an affinoid $\calP$.} If $Y=\calM(B)$, $X=\calM(\calA)$ and $Y'=\calM(\calB')$ are affinoid then by Theorem~\ref{localpinch} there exists an affinoid neighborhood $X_\veps=\calM(\calA_\veps)$ of $Y$ in $X$ such that $(Y;X_\veps,Y')$ is strongly affinoid. It remains to use Step 1.

Step 3. {\it The case of affinoid $Y$ and $Y'$.} Fix for a while a point $y'\in Y'$ and consider the fiber $h^{-1}(y')=\{y_1\..y_n\}$. Since $X$ is Hausdorff, there exists pairwise disjoint neighborhoods $W_i$ of $y_i$. Since $Y\into X$ is a closed immersion, the map of germ reductions $\tilY_{y_i}\into\tilX_{y_i}$ is an isomorphism for each $i$. By \cite[Theorem~5.1]{Temkin-local-properties} $\tilX_{y_i}$ is affine and $X_{y_i}$ is good, so shrinking $W_i$'s we can assume that they are affinoid.

Note that for a small enough Laurent neighborhood $Y'_0$ of $y'$ in $Y'$, its preimage $Y_0=Y'_0\times_{Y'}Y$ splits as $Y_0=\coprod_{i=1}^n Y_i$, where each $Y_i$ is an affinoid domain containing $y_i$. Furthermore, shrinking $Y'_0$ if necessary we can achieve that $Y_i\subseteq W_i$. Each $Y_i$ is a Weierstrass domain of the Laurent domain $Y_0$ of $Y$, hence $Y_i$ is a rational domain in $Y$. It follows that $Y_i$ is also a rational domain in $W_i\times_XY$ given by the same inequalities $|f_1|\le r_1|g|\..|f_m|\le r_m|g|$ with $f_j,g\in\calB$. Lifting $f_j|_{W_i\times_XY}$ and $g|_{W_i\times_XY}$ to functions $F_j,G$ on $W_i$ we obtain a rational domain $X_i=W_i\{r^{-1}\frac{F}{G}\}$ such that $X_i\times_XY=Y_i$. Setting $X_0=\coprod_{i=1}^nX_i$ we obtain an affinoid subdomain $\calP_0=(Y_0;X_0,Y'_0)$ of $\calP$ such that $Y'_0$ is a neighborhood of $y'$.

Since $Y'$ is compact, varying $y'$ we can find a finite set of affinoid subdata $\calP_1\..\calP_m$ of $\calP$ such that $Y'=\cup_{j=1}^mY'_j$ and hence also $Y_j=\cup_{j=1}^mY_j$. Note also that the domain $X_0=\cup_{j=1}^mX_j$ is a neighborhood of $Y$ in $X$. By step 2 each $\calP_j$ possesses a strongly affinoid covering, hence $\calP_0=(Y;X_0,Y')$ possesses such a cover too, and it remains to use step 1.

Step 4. {\it The general case.} Choose an admissible affinoid covering $Y'=\cup_{i\in I} Y'_i$. Then $Y_i=Y\times_{Y'}Y'_i$ form an admissible affinoid covering of $Y$. Lift $Y_i$ to an analytic compact domain $X_i\subseteq X$, then $X_0=\cup_i X_i$ is a neighborhood of $Y$ in $X$. If $W$ is a sufficiently small neighborhood of $Y$ in $W$, then $W_i=X_i\cap W$ form an admissible covering of $W$. Each $\calP_i=(Y_i;W_i,Y'_i)$ possesses a strongly affinoid covering by step 3, hence the lemma follows by applying step 1 to the datum $\calP'=(Y;W,Y')$.
\end{proof}

\begin{rem}
The assumption that $X$ is Hausdorff is necessary. For example, the lemma fails for the datum $(Y;X,Y')$, where $X$ is the unit disc with the doubled origin, $Y=\calM(k)\coprod\calM(k)$ is the doubled origin and $Y'=\calM(k)$.
\end{rem}

\subsection{Existence of pinchings}
If a pushout $X'=X\coprod_YY'$ exists we will use the notation $X'=\coprod\calP$ and denote the induced morphisms $i'\:Y'\to X'$ and $h'\:X'\to X$. Any $X'$-space $X'_0$ induces the pullback pushout $\calP_0=\calP\times_{X'}X'_0$. Here is our main result about pinchings. We follow the terminology of \cite{maculanpoineua} where a compact Stein space is called a {\em Liu space}.

\begin{theor}\label{pinchth}
(i) Existence: any Hausdorff pinching datum $\calP=(Y;X,Y')$ possesses a pushout $X'=\coprod\calP$.

(ii) Compatibility with topologies and structure sheaves: $X'=X\coprod_{Y}Y'$, $V\subseteq X'$ is an analytic domain if and only if its preimages in $X$ and $Y'$ are analytic domains and $$\calO_{X'_G}=i'_*\calO_{Y'_G}\times_{i_*h'_*\calO_{Y_G}}h'_*\calO_{X_G}.$$

(iii) Bicartesianity: $i'$ is a closed immersion, $h'$ is finite and $Y=X\times_{X'}Y'$.

(iv) The affinoid and Liu cases: $X'$ is a Liu space if and only if $\calP$ is a Liu space, and $X'$ is affinoid if and only if $\calP$ is strongly affinoid.

(v) Flat sites and uniformity: the pushout and pullback functors $\calP_0\mapsto\coprod\calP_0$ and $X'_0\mapsto \calP\times_{X'}X'_0$ are inverse equivalences between the categories of $\calP$-flat Hausdorff $\calP$-data $\calP_0$ and $X'$-flat Hausdorff $X'$-spaces $X'_0$.

(vi) The equivalence from (v) respects the following properties: smooth, embedding of a subdomain and the properties (i)--(xx) from \cite[Theorem~1.2]{descent}, including quasi-smooth, quasi-\'etale, \'etale, finite, proper, etc., (flatness and topological separatedness are automatically assumed on both sides, so the claim in their case is vacuous).
\end{theor}
\begin{proof}
First, let us construct a space accordingly to (ii). Consider the topological pushout $X'=X\coprod_YY'$, provide it with the pushout $G$-topology: $V\subseteq X'$ is an analytic subdomain if and only if its preimages in $X$ and $Y'$ are analytic subdomains, declare $V$ affinoid if its preimage $\calP_V=(Y_V;X_V,Y'_V)$ in $\calP$ is strongly affinoid and use the homeomorphism $V=\calM(\calA_V\times_{\calB_V}\calB'_V)$ from Lemma~\ref{affpinch}. It follows from Lemma~\ref{pinchcover} that strongly affinoid covers of $\calP$ form a net, therefore the so-defined affinoid domains form a net in $X'$, and it is now a routine check that $X'$ with this net of affinoid domains is an analytic space.

Now, let us check that $X'$ is the pushout in the category of $k$-analytic spaces. Assume that $X\to T$ and $Y'\to T$ are morphisms that agree on $Y$, and let us prove that they factor through $X'$ uniquely. Choose an admissible affinoid covering $T=\cup_i T_i$. By Lemma~\ref{pinchcover} its preimage to $\calP$ possesses a strongly affinoid refinement $\calP=\cup_j\calP_j$ with each $\calP_j=(Y_j;X_j,Y'_j)$ mapping to some $T_{i(j)}$. The corresponding affinoid domain $X'_j\subseteq X'$ is the pushout of $\calP_j$ in the category of $k$-affinoid spaces, hence the morphism $\calP_j\to T_{i(j)}$ factors uniquely through $X'_j$. Working with affinoid covers of intersections it is easy to see that the induced morphisms $X'_j\to T$ are compatible, hence we obtain a unique morphism $X'\to T$ through which $\calP\to T$ factors. This finishes the proof of (i) and (ii).

The affinoid case of (v) was proved in Lemma~\ref{affpinch}(ii) and the general case follows because both the pushout and the pullback are constructed locally -- one glues pushouts (resp. pullbacks) of strongly affinoid data (resp. affinoid spaces). Similarly, (iii) follows from Lemma~\ref{affpinch}(i).

Since the morphism $X\to X'$ is finite, $X'$ is a Liu space if and only if $X$ is a Liu space, and then also the closed subspaces $Y$ and $Y'$ are Liu. Similarly, if $X'$ is affinoid, then $\calP$ is affinoid, and it is strongly affinoid because $\calO_{X'}(X')=\calO_X(X)\times_{\calO_Y(Y)}\calO_{Y'}(Y')$ by the pushout property. Conversely, if $\calP$ is strongly affinoid then the pushout is affinoid due to the construction in the first paragraph.

Finally, the properties in (vi) are stable under base change, hence only descent should be established. A subdomain is a flat monomorphism and a smooth morphism is a boundaryless quasi-smooth one, so we should only deal with properties (i)--(xx). Properties (i)--(iv) follow from the fact that $f\:Y\coprod X'\to X$ is surjective, see \cite[Theorem~3.4]{descent}, and properties (v)--(xiv) follow from the fact that $f$ is also finite, and hence $G$-surjective and properly surjective, see \cite[Theorems~3.8 and 3.12]{descent}. It remains to recall that the descent of flatness is provided by claim (v), and as in the proof of \cite[Theorem~4.7]{descent} the remaining properties are combinations of flatness with the properties whose descent was already established.
\end{proof}

\section{Liu algebras and spaces}\label{conjsec}

\subsection{Liu algebras}
Affinoid algebras are Banach-theoretic analogues of finitely generated $k$-algebras. A foundational result in algebraic geometry states that $A$ is finitely generated over $k$ whenever it is locally finitely generated over $k$: there exist elements $f_1\..f_n\in A$ generating the unit ideal and such that the localizations $A_{f_i}$ are finitely generated over $k$. By a {\em localization} of a Banach $k$-algebra $\calA$ we mean an algebra of the form $\calA'=\calA\{r^{-1}\frac{f}{g}\}$, where $f=(f_1\..f_n)$ and $g$ generate the unit ideal. Then $X'=\calM(\calA')$ is a subset of $X=\calM(\calA)$ called a rational subset.

\begin{defin}
A Banach $k$-algebra $\calA$ is called {\em locally $k$-affinoid algebra} if there exist localizations $\calA_i$ of $\calA$ such that each $\calA_i$ is $k$-affinoid, the spaces $X_i=\calM(\calA_i)$ cover $X=\calM(\calA)$ and the map $\phi\:\calA\to\prod_i\calA_i$ is universally injective in the following sense: for any localization $\calA'=\calA\{r^{-1}\frac{f}{g}\}$ the base change map $\phi\ \wtimes_\calA\calA'$ is injective.
\end{defin}

\begin{rem}
(i) Unlike the usual commutative algebra, the injectivity of $\phi$ is not automatic and there might exist exotic nilpotent elements killed in all localizations. This is one of mechanisms in which the structure presheaf on $X$ can fail to be a sheaf. It is not clear if universal injectivity is equivalent to injectivity, so one has to use the universal version of $\phi$ in the definition.

(ii) Similarly to \cite[Lemma~8.2.2/2]{bgr} it is easy to see that any cover by rational subsets possesses a refinement of a very special form, called a {\em rational cover}: it is determined by elements $f_1\..f_n$ generating the unit ideal and numbers $r_1\..r_n$ by $\calA_i=\calA\{\frac{r_i}{r_1}\frac{f_1}{f_i}\..\frac{r_i}{r_n}\frac{f_n}{f_i}\}$.

(iii) Using Theorem \ref{localpinch} it is easy to see that for any affinoid pinching datum $(\calB;\calA,\calB')$ the fiber product $\calA'$ is a locally affinoid algebra. Liu's example, shows that it does not has to be affinoid.
\end{rem}

Note that the spectrum of a Liu algebra possesses a natural structure of a $k$-analytic space glued from affinoid rational domains. Following the method of \cite[\S8.2]{bgr} M. Xia has recently obtained the following generalization of Tate's acyclicity for these spaces (see the proof of \cite[Theorem~B.4]{Xia}):

\begin{theor}[Xia]\label{tateconj}
Given a locally affinoid algebra $\calA$ provide $X=\calM(\calA)$ with the natural structure of the $k$-analytic space. Then for any rational subset $X'=\calM(\calA')$ in $X$ one has that $\calO_X(X')=\calA'$ and for any finite covering of $X$ by rational domains the \v{C}ech complex is acyclic.
\end{theor}

Also, it seems plausible that using the methods of \cite[\S6.4]{bgr} one can prove that, similarly to some results of Section~\ref{affsec}, finiteness of reduction should provide a criterion for Liu algebra to be affinoid:

\begin{conj}
Assume that $k$ is stable and $\calA$ is a Liu algebra over $k$, then

(i) $\calA$ is affinoid if and only if $\tilcalA_\gr$ is a finitely generated $\tilk_\gr$-algebra.

(ii) If $\calM(\calA)$ is strictly $k$-analytic, then $\calA$ is affinoid if and only if $\tilcalA$ is finitely generated over $\tilk$.
\end{conj}

\subsection{Liu spaces}
Liu studied compact Stein spaces (in rigid-theoretic setting), and his results were extended to general analytic spaces by Maculan and Poineau. In
\cite[Theorem~1.11]{maculanpoineua} they proved that a compact space $X$ has trivial higher coherent cohomology if and only if it is separated, has enough global functions to separate elements and the structure sheaf has trivial higher cohomology, and coined the term {\em Liu space} for such spaces $X$. The algebra $\calA=\calO_X(X)$ is then called a {\em Liu algebra}. By \cite[Proposition~3.12]{maculanpoineua} any Liu space $X$ possesses a cover by rational domains, which are affinoid, and by \cite[Corollary~3.17]{maculanpoineua} $X=\calM(\calA)$, where $\calA=\calO_X(X)$. Therefore, $\calA$ is locally affinoid and $X$ is its spectrum.

Conversely, Theorem~\ref{tateconj} implies that for any locally affinoid algebra $\calA$ the spectrum $X=\calM(\calA)$ has an acyclic structure sheaf and hence is a Liu space. This yields \cite[Theorem~B.4]{Xia}: a $k$-Banach algebra is locally affinoid if and only if it is Liu. In particular, Serre's criterion works once one replaces affinoid spaces by Liu spaces. From this point of view, the analogy with the category of varieties is much more complete, and the main difference is that any variety, which is an affine scheme $\Spec(A)$, is an affine variety (that is $A$ is finitely generated), while Liu's spaces are generalized affinoid objects, whose algebra is not topologically finitely generated.

\bibliographystyle{amsalpha}
\bibliography{ap}

\end{document}